\documentclass[12pt]{amsart}
\usepackage{ amsmath, amsthm, amsfonts, amssymb, color}
 \usepackage[colorlinks, citecolor=blue,pagebackref,hypertexnames=false]{hyperref}
 \allowdisplaybreaks
\begin{document}

\newtheorem{theorem}{Theorem}[section]
\newtheorem{proposition}[theorem]{Proposition}
\newtheorem{coro}[theorem]{Corollary}
\newtheorem{lemma}[theorem]{Lemma}
\newtheorem{definition}[theorem]{Definition}
\newtheorem{example}[theorem]{Example}
\newtheorem{remark}[theorem]{Remark}
\newcommand{\ra}{\rightarrow}
\renewcommand{\theequation}
{\thesection.\arabic{equation}}
\newcommand{\ccc}{{\mathcal C}}
\newcommand{\one}{1\hspace{-4.5pt}1}

 \def \Lips  {{   \Lambda}_{L}^{ \alpha,  s }(X)}
\def\BL {{\rm BMO}_{L}(X)}
\def\HAL { H^p_{L,{at}, M}(X) }
\def\HML { H^p_{L, {mol}, M}(X) }
\def\HM{ H^p_{L, {mol}, 1}(X) }
\def\Ma { {\mathcal M} }
\def\MM { {\mathcal M}_0^{p, 2, M, \epsilon}(L) }
\def\dMM { \big({\mathcal M}_0^{p,2, M,\epsilon}(L)\big)^{\ast} }
  \def\RR {  {\mathbb R}^n}
\def\HSL { H^p_{L, S_h}(X) }
\newcommand\mcS{\mathcal{S}}
\newcommand\mcB{\mathcal{B}}
\newcommand\D{\mathcal{D}}
\newcommand\C{\mathbb{C}}
\newcommand\N{\mathbb{N}}
\newcommand\R{\mathbb{R}}
\newcommand\G{\mathbb{G}}
\newcommand\T{\mathbb{T}}
\newcommand\Z{\mathbb{Z}}

\newcommand\CC{\mathbb{C}}
\newcommand\NN{\mathbb{N}}
\newcommand\ZZ{\mathbb{Z}}

\renewcommand\Re{\operatorname{Re}}
\renewcommand\Im{\operatorname{Im}}

\newcommand{\mc}{\mathcal}

\newcommand{\eps}{\varepsilon}
\newcommand{\pl}{\partial}
\newcommand{\supp}{{\rm supp}{\hspace{.05cm}}}
\newcommand{\x}{\times}
\newcommand{\mar}[1]{{\marginpar{\sffamily{\scriptsize
        #1}}}}
\newcommand{\as}[1]{{\mar{AS:#1}}}

\newcommand\wrt{\,{\rm d}}


\title[Spectral multipliers via resolvent type estimates]
{Spectral multipliers via resolvent type estimates  on non-homogeneous metric measure spaces}

\author{Peng Chen, \ Adam Sikora, \ and \ Lixin Yan}
\address{Peng Chen, Department of Mathematics, Sun Yat-sen (Zhongshan)
University, Guangzhou, 510275, P.R. China}
\email{chenpeng3@mail.sysu.edu.cn}
\address{
Adam Sikora, Department of Mathematics, Macquarie University, NSW 2109, Australia}
\email{adam.sikora@mq.edu.au}
\address{
Lixin Yan, Department of Mathematics, Sun Yat-sen (Zhongshan) University, Guangzhou, 510275, P.R. China}
\email{mcsylx@mail.sysu.edu.cn
}

\date{\today}
\subjclass[2000]{42B15, 42B20,   47F05.}
\keywords{Spectral multipliers, non-homogeneous type spaces, finite propagation speed property, $L^p$ spectrum independence, manifolds with ends, Schr\"odinger  operators. }

\begin{abstract}
We describe a simple but surprisingly effective technique of obtaining  spectral multiplier results for 
abstract operators which satisfy the  finite propagation speed property for the corresponding wave equation propagator. 
We show that, in this setting, spectral multipliers follow from resolvent type estimates.
 The most notable point of the paper is that our approach   is very flexible and can be
applied  even if the corresponding 
ambient space does not satisfy the doubling condition or if the semigroup generated by an operator 
is not uniformly bounded. As a corollary we obtain the $L^p$ spectrum independence for several second
order differential operators and recover some known results. Our examples include
 the Laplace-Belltrami operator on manifolds with ends and Schr\"odinger operators with strongly 
subcritical potentials.  
\end{abstract}

\maketitle


\section{Introduction}
\setcounter{equation}{0}

This paper is devoted to spectral multiplier theory, which is one of most significant areas  of Harmonic 
Analysis. In the  most general framework one considers a self-adjoint, non-negative usually unbounded 
operator $L$ acting on space $L^2(X,\mu)$ where $(X, d, \mu)$ is a metric measure space with metric
 $d$ and measure $\mu$. 

By spectral theory for any bounded Borel function $F$ one can define
the operator
$$
F(L)=\int_{0}^\infty F(\lambda) dE_{L}(\lambda)
$$
where $dE_{L}$ is a spectral resolution of the operator $L$. The $L^2(X)$ norm of the operator $F(L)$ is bounded 
by the $L^\infty$ norm of $F$. Then spectral multiplier theory asks what sufficient conditions are required to ensure that
the operator $F(L)$ extends to a bounded operator acting on spaces $L^p(X)$ for all $ 1\le p \le \infty$ 
or for some smaller  range of exponents $p$. The question is motivated by the problem of convergence of 
eigenfunction expansion and includes the Bochner-Riesz means analysis. 
This area  was at first inspired by celebrated  Fourier multiplier results of Mikhlin and H\"ormander, 
see \cite{Hor, Mik}, which initiated stating sufficient conditions for function $F$ in terms of its differentiability.

The spectral multiplier theory is   well developed and understood. However almost all results use the standard 
assumptions which include the doubling condition and uniform estimates for the corresponding 
semigroup. These two types of conditions are not completely natural and in fact there are many significant 
examples  which do not satisfy these assumptions and  we would like to investigate them.
 These two situations are  main points of interest of our approach.  {\it Firstly} we are able to study the 
 ambient spaces $(X, d, \mu)$ which do not satisfy the doubling condition, see \eqref{dc} and \eqref{vd} below. 
 In particular, we obtain a spectral multiplier result for Laplace-Beltrami operator acting on the class of 
 manifolds with ends described by  Grigor'yan and Saloff-Coste in \cite{GS}. The doubling condition usually fails 
 for the manifolds with ends. {\it A second point} is that we are able to treat   the case of operators for 
 which the corresponding semigroup does not 
satisfy uniform bounds for large time in the full range of $L^p$ spaces.  Examples of such situation come from 
investigation of strongly subcritical Schr\"odinger semigroups discussed by  Davies and Simon in \cite{DS} and by Murata in 
\cite{M2}. Investigation of Hodge Laplacian operators also often leads to semigroups which are not uniformly bounded on  $L^p$ spaces, see \cite{CouZ} and Remark~\ref{rm5.19} below. To the best of our knowledge the spectral multiplier results have not been studied before in the 
setting of manifolds with ends and strongly subcritical Schr\"odinger operators and our approach leads to 
first examples of such results.

We also would like to mention some other interesting features of the spectral multiplier techniques 
which we develop here. Our approach yields an alternative proof of the classical spectral multiplier 
result  for operators which generate semigroups with the standard Gaussian bounds, see e.g. 
\cite{COSY, CowS, DOS} and Remark \ref{rm5.112} below. 
Our discussion provides also alternative proofs of most of  Davies' results from \cite{Da} and 
improves some of estimates stated there. Related to this point is the issue of $L^p$ spectral independence
 of the operators which we study here. There are well-known
and important examples of operators of a similar nature to the ones which we investigate here, 
 where their spectra  depend upon the space $L^p(X,\mu)$ (see \cite{A, DST}).
 It is natural to ask whether the spectra of operators which we consider here also depend upon $p$.
Examples of results concerning the $L^p$ independence are described in  \cite{Da1, Da}. See also  the 
references therein.
As a corollary of our main result we prove the $L^p$ spectral independence in all settings which we 
consider here including the manifolds with ends. 

One of main assumptions of our results is the finite propagation speed for the corresponding wave 
equation propagator. We describe this standard notion in the next section. 
To be able to state our main result we define the volume function $V_r$ by the formula 
$$
V_r(x)=V(x,r) = \mu(B(x,r)),
$$
where $B(x,r)$ is a ball with radius $r$ and  center at $x$. 
Next, for a function $W \colon X \to \R$ we set $M_W$ to be operator of multiplication by W, that 
is $$M_W f(x)=W(x)f(x).$$ To simplify notion,  in what follows, we identify $W$ and the operator $M_W$ 
that is we denote $M_WT$ by $WT$ for any linear operator $T$. 

\smallskip

 Now our main result can be formulated in the following way.
 \begin{theorem}\label{thm1.1}
Suppose that $L$ a self-adjoint, non-negative operator which satisfies the finite propagation  speed property for
the corresponding wave equation.
Suppose next that for some exponents $\sigma>0$  and $\kappa \ge 0$
\begin{equation}\label{a1}\tag{$R_{\sigma, \kappa}$}
\|V^{1/2}_t(I+t^2L)^{-\sigma}\|_{2 \to \infty} \le C (1+t^2)^\kappa
\end{equation}
for all $t>0$.  Then

 i) There exists a constant $C>0$ such that
 \begin{eqnarray}\label{e3.55}
\|e^{i\xi tL}e^{-t L}\|_{1\to 1}\leq C(1+\xi^2)^{\sigma +\kappa+1/4} (1+t)^{\kappa}
\end{eqnarray}
for all $t>0$ and $\xi\in{\R}$;

ii) For  a bounded Borel function $F$ such that $\supp F\subset [-1, 1] $
and $F\in  H^s(\mathbb R)$  for some $s>2\sigma+2\kappa+1$,
the operator $F(L)$ is bounded on $L^p(X)$ for all $1\leq p\leq \infty$, and
 there exists constant $C=C(s)>0$ such that
 \begin{eqnarray}\label{e3.5555}
\|F(tL)\|_{p\to p} \leq C(1+t)^{\kappa}\|F\|_{H^{s}}
\end{eqnarray}
for all $t>0$.
\end{theorem}

In Section \ref{sec5} we describe a number of applications of Theorem~\ref{thm1.1}, including the manifolds 
with ends and strongly subcritical Schr\"odinger operators,  which explain  significance of the above statement and rationale for assuming resolvent type estimate \eqref{a1}. Additional rationale for condition \eqref{a1} can be found in \cite[Proposition 2.3.4]{BCS}.

The comprehensive  list of relevant literature concerning spectral multipliers is enormous and too long to be included in this note  so we just refer readers to \cite{COSY, DOS}  and
references within as a possible  starting point for gathering complete  bibliography of the subject.

\section{Doubling condition and finite  propagation speed property}
\setcounter{equation}{0}

{\bf Doubling condition.} In our approach the doubling condition does not play essential role. 
However, because we use this notion in the discussion or in our results we recall it here. 
We say that metric measure space $(X, d, \mu)$ satisfies the doubling condition if there exists constant $C$ such that
\begin{equation}  \label{dc}
V(x,2r) \le CV(x,r)
\end{equation}
for all $x\in X$ and $r>0$. As a consequence,  there exist constants $C,n>0$ such that
\begin{equation}  \label{vd}
V(x, s r)\leq Cs^n V(x,r) \,\,\mbox{for all}\,\, s \geq 1,  r>0.
\end{equation}
Note that the doubling condition fails in the example which we consider in Section \ref{sub5.1} and it 
may or may not  be satisfied by $(X, d, \mu)$ in Section \ref{sub5.2}. 
Metric measure spaces which satisfy the doubling 
condition are often called {\em homogenous spaces}. 

The following statement provides another rationale for condition \eqref{a1} and will be used in 
Sections \ref{sub5.3} and \ref{sub5.4}. 

\begin{lemma}\label{prop2}
Assume that space $(X,d,\mu)$ satisfies the doubling condition \eqref{vd} with exponent $n$.
Then for all $\sigma>n/4$ and $\kappa\geq 0$, condition  \eqref{a1} is equivalent with the following estimate
\begin{equation}\label{a22}
\|V^{1/2}_te^{-t^2L}\|_{2 \to \infty} \le C (1+t^2)^\kappa.
\end{equation}
\end{lemma}
\begin{proof}
The proof is very simple. For example it is just a minor modification of the proof of Proposition~2.3.4 of \cite{BCS}.
\end{proof}

{\bf Finite  propagation speed.} Let $L$ be a self-adjoint non-negative operator acting
on $L^2(X).$
Following  \cite{CouS},  we say that $(X,d,\mu,L)$, or in short $L$, satisfies the finite propagation  speed property for
 the corresponding wave equation propagator if
\begin{equation}  \label{fs11}
\langle \cos(r\sqrt L) f_1, f_2 \rangle = 0
\end{equation}
for all $f_i\in L^2(B_i,\mu)$, $i=1,2$, where $B_i $ are  open balls in $X$ such that  $d(B_1,B_2)>r>0$.

We shall use the following notational convention, see \cite{Si}.  For $r>0$, set
\begin{equation*}
D_r=\{ (x,\, y)\in X\times X: {d}(x,\, y) \le r \}.
\end{equation*}
Given a linear operator $T$ from $L^2(X,\mu)$  to $L^2(X,\mu)$, we   write
\begin{equation}\label{e2.1}
\text{supp}\,T  \subseteq D_r
\end{equation}
 if $\langle T f_1, f_2 \rangle = 0$ whenever  $f_1\in L^2(B_1,\mu)$, $f_2\in L^{2}(B_2,\mu)$,
 and $B_1,B_2$ are balls such that
 $d(B_1,B_2)> r$.   Note that  if $T$ is an integral operator with   kernel
$K_T$, then \eqref{e2.1} coincides with the  standard meaning of
$\text{supp} \, K_{T}  \subseteq D_r$,
 that is $K_T(x, \, y)=0$ for all $(x, \, y) \notin D_r$. Now we can state the finite  propagation speed
 property \eqref{fs11} in the following way
 \begin{equation}\label{fs}\tag{FS}
\text{supp} \, \cos(r\sqrt L) \subseteq D_r,\quad\forall\,r>0.
\end{equation}

Property \eqref{fs} holds for most of second order self-adjoint operators  and  it
is equivalent to Davies-Gaffney
estimates, see estimate \eqref{dg} below and \cite{Si, CouS}. See also  \cite{CGT}
for earlier examples of the finite propagation speed property technique. Examples of application of the  finite speed
propagation techniques to the spectral multiplier theory can be found in \cite{CowS, COSY}. 

In what follows we will need the following straightforward observation, see\cite{COSY, Si}. 

  \begin{lemma}\label{nos}
Assume that $L$ satisfies the finite propagation  speed property \eqref{fs} for
the corresponding wave equation.
  Let $\Phi\in L^1({\R})$ be  an even function such that
 $\mbox{\em supp}\,\widehat\Phi \subset [-1,1].$  Then 
  \begin{equation}
 \mbox{\em supp} \, {\Phi(r\sqrt L)} \subseteq D_r
\label{noseq}
\end{equation}
 for all $r>0$.
  \end{lemma}
  \begin{proof}
If $\Phi$ is an even function, then by the Fourier inversion formula,
\begin{equation*}\label{formu}
\Phi(r\sqrt L) =\frac{1}{2\pi}\int_{-\infty}^{+\infty}%
  \widehat{\Phi}(s) \cos(rs\sqrt L) \;ds.
\end{equation*}
However,  $\supp\hat{\Phi} \subseteq [-1, 1]$
so the lemma follows from \eqref{fs}.
\end{proof}

 \section{Spectral multiplier theorems - the proof of the main result}
\setcounter{equation}{0}

This section is devoted to the proof of  Theorem \ref{thm1.1}, which is based on the wave equation technique.

\begin{proof}[Proof of Theorem \ref{thm1.1} ]
First, we prove estimate~\eqref{e3.55}.
A direct calculation as in \cite[page 360]{BCS} shows that for all $a>0$, $x\in {\R}$,
$$
\frac{1}{\Gamma(a)}\int_0^\infty(s-x^2)^a_+e^{-s}\,ds=e^{-x^2},
$$
where
$$
(t)_+=t \quad \mbox{if} \quad t \ge 0 \quad \quad \mbox{and}\quad  (t)_+=0 \quad  \mbox{if} \quad
t<0.
$$
Hence
$$
C_a\int_0^\infty\left(1-\frac{x^2}{s}\right)^a_+e^{-{s\over 4}}s^a \,ds=e^{-{x^2\over 4}}
$$
for some suitable $C_a>0$.
Taking the Fourier transform of both sides of the above inequality yields
$$
\int_0^\infty F_a(\sqrt s\xi)  s^{a+\frac{1}{2}}e^{-{s\over 4}}ds=e^{-\xi^2},
$$
where $F_a$ is the Fourier transform of the function $t \to (1-{t^2})^a_+$ multiplied
by the appropriate constant. Hence, by spectral theory for $\nu>0$
$$
e^{-\nu L}=\int_0^\infty F_a(\sqrt{s\nu L}\,)  s^{a+\frac{1}{2}}e^{-{s\over 4}}ds.
$$
Thus we can rewrite $e^{i\xi t L}e^{-t L}$ as
\begin{eqnarray}\label{e4.11}
e^{i\xi t L}e^{-t L}=\int_0^\infty F_a(\sqrt{sL}\,)
s^{a+\frac{1}{2}}(t-i\xi t)^{-a-\frac{3}{2}}\exp\left({-\frac{s}{4(t-i\xi t)}}\right) ds.
\end{eqnarray}
Note that $\supp \widehat{F_a}\subset [-1,1]$. By Lemma \ref{nos}
\begin{equation*}
 \mbox{
  supp} \, {F_a(s\sqrt {L}\,) } \subseteq D_s,\ \forall
 \,s>0.
 \end{equation*}
Thus
\begin{eqnarray}\label{e4.33}
\|F_a(\sqrt{sL})\|_{1\to 1}& = &\sup_y\int_X |K_{F_a(\sqrt{sL}\,) }(x,y)|d\mu(x)\nonumber\\
&=&\sup_y\int_{B(y,\sqrt s)} |K_{F_a(\sqrt{sL}\,) }(x,y)|d\mu(x)\nonumber\\
&\leq & \sup_y \mu(B(y,\sqrt{s}))^{1/2} \left(    \int_{X} |K_{F_a(\sqrt{sL}\,) }(x,y)|^2d\mu(x)\right)^{1/2}\nonumber\\
&=& \|V^{1/2}_{\sqrt{s}}F_a(\sqrt{sL}\,) \|_{2\to \infty}\\
&\leq& \|V^{1/2}_{\sqrt{s}}(1+sL)^{-\sigma}\|_{2\to \infty}
\|(1+sL)^{\sigma}F_a(\sqrt{sL}\,) \|_{2\to 2}\nonumber\\
&\leq& \|V^{1/2}_{\sqrt{s}}(1+sL)^{-\sigma}\|_{2\to \infty}
\|(1+\lambda^2)^{\sigma}F_a(\lambda\,) \|_{L^\infty}\nonumber.
\end{eqnarray}
Next we estimate the second term of the last inequality above that is  $L^\infty$ norm of $(1+\lambda^2)^{\sigma}F_a(\lambda\,)$. For this purpose we recall the
following well-known asymptotic for Bessel type functions, often used in discussion of 
the kernel of the standard  Bochner-Riesz operators, see e.g.
 page  391 of \cite{St}: 
 $F_a(\lambda)$ is bounded and has the following asymptotic expansion
 \begin{eqnarray}\label{e2.7}
 F_a(\lambda)
 \sim |\lambda|^{-1-a}\left[e^{2\pi i|\lambda|}
 \sum_{j=0}^\infty \alpha_j |\lambda|^{-j}+e^{-2\pi i|\lambda|}
 \sum_{j=0}^\infty \beta_j |\lambda|^{-j}\right]
 \end{eqnarray}
 as $|\lambda|\to \infty$, for suitable constants $\alpha_j$ and $\beta_j$.
 
Hence  $\|(1+\lambda^2)^{\sigma}F_a(\lambda\,)\|_{\infty} < \infty$ as soon as $2\sigma-1\leq a$.
Now if $\kappa$ is an exponent from \eqref{a1} and  $2\sigma-1\leq a$ then by  \eqref{e4.33} 
\begin{eqnarray}\label{e4.22}
\|F_a(\sqrt{sL}\,) \|_{1\to 1}
  \leq  C(1+ s)^{\kappa}.
  \end{eqnarray}
By \eqref{e4.11} and \eqref{e4.22}
\begin{eqnarray*}
\|e^{it\xi L}e^{-t L}\|_{1\to 1}
&\leq & \int_0^\infty (1+s)^{\kappa}
s^{a+\frac{1}{2}}(t^2+(\xi t)^2)^{-{a\over 2}-\frac{3}{4}}\exp\left({-\frac{s}{4(t+t \xi^2)}}\right)ds\\
&\leq & C(1+t)^{\kappa} (1+\xi^2)^{{a\over 2}+{\kappa}+\frac{3}{4} }.
\end{eqnarray*}
Finally,  noting that  $2\sigma-1\leq a$, we obtain \eqref{e3.55}.

To prove~\eqref{e3.5555}
we write $G(\lambda)=F(\lambda)e^{\lambda}$, and then  $F(tL)=G(tL)e^{-tL}.$  Hence,
$$
F(tL)=G(tL)e^{-tL}=\int_{\mathbb{R}} \widehat{G}(\xi) e^{it\xi L}e^{-tL}\,d\xi.
$$
Note that $s>2\sigma+2\kappa+1$ so by \eqref{e3.55}
\begin{eqnarray*}
\|F(tL)\|_{1\to 1}&\leq& \int_{\mathbb{R}} |\widehat{G}(\xi)| \|e^{it\xi L}e^{-tL}\|_{1\to 1}\,d\xi\\
&\leq& C\int_{\mathbb{R}} |\widehat{G}(\xi)|(1+\xi^2)^{\sigma +\kappa+1/4} (1+t)^{\kappa}\,d\xi\\
&\leq&C (1+t)^{\kappa}\|G\|_{H^s} \left(\int_{\mathbb{R}} (1+\xi^2)^{\sigma +\kappa+1/4-s/2}\,d\xi\right)^{1/2}\\
&\leq& C(1+t)^{\kappa}\|G\|_{H^s}.
\end{eqnarray*}
Observe that $\supp F\subset [-1, 1]$ so  $ \|G\|_{H^s}\leq C\|F\|_{H^s}$ and 
$$
\|F(tL)\|_{1\to 1} \leq C_s(1+t)^{\kappa}\|G\|_{H^s}\leq C_s(1+t)^{\kappa}\|F\|_{H^s}.
$$
Since $F$ is a bounded Borel function, $F(L)$ is bounded on $L^2(X)$. By interpolation and duality,
$F(L)$ is bounded on $L^p(X)$ for $1\leq p\leq\infty.$
This ends the proof of Theorem~\ref{thm1.1}.
\end{proof}

In the next section, in our discussion of spectral independence of operator $L$ we shall use 
the following corollary of Theorem~\ref{thm1.1}  which states the spectral multiplier result
 for non-compactly supported functions. Recall that $[x]$ stands for the integer part of $x\in \R$.  
\begin{coro}\label{coro1.1}
Suppose that $L$ satisfies the finite propagation  speed property for
the corresponding wave equation and that  for some $\sigma>0$ ,$\kappa \ge 0$ resolvent type 
estimate \eqref{a1} holds. 
 Next assume that  for function $F$ on $\mathbb{R}$,  there exists a constant $C>0$
such that  for $m=0, 1,, \cdots, [2\sigma+2\kappa+1]+1$,
\begin{eqnarray}\label{e1.1}
  \sup_{0\leq \lambda<1}\big|{d^{m}\over d\lambda^m} F(\lambda)\big|\leq C
\end{eqnarray}
and  for some  $\epsilon>0$
\begin{eqnarray}\label{e1.2}
\sup_{\lambda\geq 1}\big|\lambda^{m+\epsilon} {d^{m}\over d\lambda^m} F(\lambda)\big|\leq C.
\end{eqnarray}
Then the operator $F(L)$ is bounded on $L^p(X)$ for all  $1\leq p\leq\infty$.
 \end{coro}
\begin{proof}
Let   $\phi \in C_c^{\infty}(\mathbb R) $   be  a function such that
$\supp \phi\subseteq \{ \lambda: 1/4\leq |\lambda|\leq 1\}$ and
 $
\sum_{\ell\in \ZZ} \phi(2^{-\ell} \lambda)=1$ for all
${\lambda>0}.
$
Set $\phi_0(\lambda)= 1-\sum_{\ell=1}^{\infty} \phi(2^{-\ell} \lambda)$ and
$$
F(\lambda)=F(\lambda)\phi_0(\lambda)+\sum_{\ell=1}^{\infty} F(\lambda) \phi(2^{-\ell} \lambda).
$$
Set $\delta_tF(\lambda)=F(\lambda t)$. It follows from ii) of Theorem~\ref{thm1.1} that for $s>[2\sigma+2\kappa +1]+1$,
\begin{eqnarray*}
\|F(L) \phi(2^{-\ell} L)\|_{1\to 1} \leq
  (1+2^{-\ell})^{\kappa}\|\phi\delta_{2^{\ell}} F\|_{H^{s}},\ \ \ \ \ell=1,2,\cdots,
\end{eqnarray*}
which yields
\begin{eqnarray*}
\|F(L)\|_{1\to 1} \leq C\|F\phi_0\|_{H^s}
+C\sum_{\ell=1}^\infty (1+2^{-\ell})^{\kappa}\|\phi\delta_{2^{\ell}} F\|_{H^{s}}
\leq C
\end{eqnarray*}
when $F$ satisfies conditions \eqref{e1.1} and \eqref{e1.2}.

Since $F$ is a bounded Borel function, $F(L)$ is bounded on $L^2(X)$. By interpolation and duality,
$F(L)$ is bounded on $L^p(X)$ for $1\leq p\leq\infty.$
\end{proof}

 \section{$L^p$-spectral independence  }
\setcounter{equation}{0}

In this section, we assume that operator $L$ is a non-negative self-adjoint operator on
$L^2(X)$ such that $e^{-tL}$ can be extended to a strongly continuous one-parameter
semigroup on $L^p(X)$ for all $1\leq p<\infty$.
For $1\leq p<\infty$ we denote by $L_p$ the generator of the considered semigroup acting on $L^p(X)$ space 
  and by  Spec$(L_p)$ its spectrum. This means that $L_p$ and $L_q$ coincide on $L^p\cap L^q$ and so if $z \notin ${\rm Spec}$(L_p)$ and  $(z-L_p)^{-1}$ extends to a bounded operator on $ L^q(X)$
  then $z \notin ${\rm Spec}$(L_q)$. To simply notation we use the convention 
$L=L_2$ consistent with the rest of this note. 

We start with proving the following lemma. 

 \begin{lemma}\label{le4.1}
Suppose that $L$ satisfies the finite propagation  speed property for
the corresponding wave equation and that 
 for some $\sigma >0 $, $\kappa\ge 0$ the resolvent type estimate \eqref{a1} holds. 
Then 
 \begin{eqnarray}\label{e4.1}
\|(z^2+L_p)^{-1}\|_{p\to p}
\leq C\left(\frac{|z|}{|{\rm Re}\, z|}\right)^{2\sigma+2\kappa+3/2} \frac{1}{|z|^2}\left(1+\frac{1}{|z|^2}\right)^{\kappa}
\end{eqnarray}
  for  all ${\rm Re}\, z > 0$ and  $1\leq p<\infty.$
  
  In addition if $z \notin ${\rm Spec}$(L_p)$ and $z \notin ${\rm Spec}$(L_q)$ then the resolvents 
  $(z-L_p)^{-1}$ and $(z-L_q)^{-1}$ are consistent that is they coincide on $L^p(X)\cap L^q(X)$. 
 \end{lemma}

 \begin{proof}
Write $z=re^{i\theta}$ for $r>0$ and $-\pi/2<\theta<\pi/2$. By a standard formula, see e.g. \cite{Da} or 
\cite[Proposition 7.9]{O}
$$
(L_p+z^2)^{-1}=(L_p+r^2e^{2i\theta})^{-1}=e^{-i\theta}\int_0^\infty
\exp[-L_pse^{-i\theta}-sr^2e^{i\theta}]ds.
$$
Note that the above relation  shows consistency of the resolvents.
Next we set $t=r\cos \theta$ and $\xi=\tan \theta$. Then by the above formula and \eqref{e3.55},
\begin{eqnarray*}
\|(L_p+z^2)^{-1}\|_{p\to p}
&\leq & \int_0^\infty \|e^{is\sin\theta L_p}e^{-s\cos\theta L_p}\|_{p\to p}
e^{-sr^2\cos\theta}ds\\
&= & \int_0^\infty \|e^{i\xi t L}e^{-t L}\|_{p\to p}
e^{-sr^2\cos\theta}ds\\
&\leq &C \int_0^\infty (1+\xi^2)^{\sigma+\kappa+1/4}(1+t)^{\kappa}
e^{-sr^2\cos\theta}ds\\
&\leq & C(1+\tan^2\theta)^{\sigma+\kappa+1/4} \frac{1}{r^2\cos \theta}
\left(1+\frac{1}{r^{2\kappa}}\right),
\end{eqnarray*}
which implies \eqref{e4.1}. 
 \end{proof}

Our spectral independence result can be stated now in the following way. 
 \begin{theorem}\label{th4.2}
Suppose that $L$ satisfies the finite propagation  speed property for
the corresponding wave equation and that  for some $\sigma>0$, $\kappa\ge 0$ resolvent type estimate \eqref{a1} holds 
for $L$. 
Then {\rm Spec}$(L_p)$, the spectrum of the operator $L_p$, is independent of the space $L^p(X)$
for all $1 \le p < \infty$.
 \end{theorem}
\begin{proof}
By \eqref{e4.1}, for $z\notin  [0,\infty)$  the resolvent $(z-L)^{-1}$
is bounded as an operator acting on any $L^p(X)$ space for $1\leq p<\infty$ and so   Spec$(L_p) \subset \R_+$.

Now assume that  $\rho  \ge 0$ is not contained in  {\rm Spec}$(L_2)$. Then there exists $\epsilon>0$ such that
$(\rho-\epsilon, \rho+\epsilon)\cap$ {\rm Spec}$(L_2)=\emptyset$.
Now consider a smooth function $\psi \in C^{\infty}_c(\rho-\epsilon, \rho+\epsilon)$
such that $\psi(x)=1$ for $ x \in (\rho-\epsilon/2, \rho+\epsilon/2)$.
Then  $\psi(L)=0$ so
 $$
 (I-\psi(L))(\rho-L)^{-1}=(\rho-L)^{-1}.
 $$
Notice that function $g(\lambda) = (1-\psi(\lambda))(\rho-\lambda)^{-1}$ satisfies conditions
\eqref{e1.1} and \eqref{e1.2} of Corollary~\ref{coro1.1}. Hence the operator
$(I-\psi(L))(\rho-L)^{-1}= (\rho-L)^{-1}$ extends to a bounded operator acting on $L^p(X)$.
 This shows that {\rm Spec}$(L_p) \subseteq$ {\rm Spec}$(L_2)$.

 To show the opposite inclusion, assume  that  $\rho  \ge 0$ is not contained in {\rm Spec}$(L_p)$.
By duality $(\rho-L_{p'})^{-1}$ is bounded on  $L^{p'}$. Hence by consistency and interpolation   $(\rho-L)^{-1}$
is bounded on $L^2$. The proof of Theorem~\ref{th4.2} is complete.
\end{proof}

 \section{Applications  }\label{sec5}
\setcounter{equation}{0}

In this section we describe a number of applications of our main results. Note that in Section \ref{sub5.3} below estimate
\eqref{a1} cannot hold for $\kappa=0$. This is an important motivation for the type of assumptions which we consider 
in our main result.

\subsection{Manifolds with ends}\label{sub5.1}

Firstly  we want to show one can  apply Theorem~\ref{thm1.1} to the setting of manifolds with ends 
studied in details by  Grigor'yan and Saloff-Coste in \cite{GS}. The precise description of the notion of 
manifolds with ends in full generality is complex and  technical. These technical details are not relevant for the spectral multiplier technique which we describe in this note. 
Therefore we consider here only simple model case from \cite{GS}. We leave it to the interested readers to check that the approach described below can be applied to the whole class of manifolds with ends discussed there.

Let $M$ be a complete non-compact Riemannian manifold. Let $K\subset M$ be a compact set
with non-empty interior and smooth boundary such that $M\backslash K$ has $k$ connected
components $E_1,\ldots,E_k$ and each $E_i$ is non-compact. We say in such a case that $M$
has $k$ ends with respect to $K$. Here we are going to consider only the case $k=2$ that is two different ends. 
We are also  going to assume that
each $E_i$ for $i=1,2$ is isometric to the exterior of a compact set in another manifold $M_i$ described below. In such
case we write $M=M_1\sharp M_2$. Next,  fix an integer $m\ge 3$, which will
be the topological dimension of $M$, and for any integer $2 < n \le m$, define the manifold $\mathfrak{R}^n$ by
$$
 \mathfrak{R}^n=\mathbb{R}^n\times \mathbb{S}^{m-n}.
$$
The manifold $\mathfrak{R}^n$ has topological dimension $m$ but its ``dimension at infinity" is $n$
in the sense that $V(x,r)\approx r^n$  for $r\geq1$, see \cite[(1.3)]{GS}. Thus, for
different values of $n$, the manifold $\mathfrak{R}^n$ have different dimension at infinity but
the same topological dimension $m$. This enables us to consider finite connected sums of the $\mathfrak{R}^n$
and $\R^m$.
Namely consider  integers $2< n \le  m$ and define $M$ as 
$$
    M= \mathfrak{R}^{n}\sharp\,  {\R}^{m}.
$$
Next for any $x\in M$ we put 
$$
   |x|:=\sup_{z\in K}d(x,z)
$$
and we recall that $
   V(x,r)=\mu(B(x,r))$.
From the construction of the manifold $M$, we can see that
\begin{itemize}
\item[(a)] $V(x,r)\thickapprox r^m$ for all $x\in M$, when $r\leq 1$;
\item[(b)] $V(x,r)\thickapprox r^n$ for $B(x,r)\subset \mathfrak{R}^n$, when $r> 1$; and
\item[(c)] $V(x,r)\thickapprox r^m$ for $x\in \mathfrak{R}^n\backslash K$, $r>2|x|$, or $x\in \mathbb{R}^m$, $r>1$.
\end{itemize}
It is not difficult to check that if $n< m$ then  $M$ does not satisfy the doubling condition.  Indeed, consider a
family  of balls $B(x_r,r)\subset \mathfrak{R}^n$ such that $r = |x_r| \rightarrow \infty $.
Then $V(x_r,r)\thickapprox r^n$ but $V(x_r,2r)\thickapprox r^m$ and the doubling condition fails.

Let $\Delta$   be  the Laplace-Belltrami operator acting on $M$ and let  $p_t( x, y)$ be 
 the heat kernel corresponding to the heat propagator  $e^{-t\Delta}$.
In \cite{GS} Grigor'yan and Saloff-Coste establish both the global upper bound and lower
bound for the heat kernel of the semigroup $e^{-t\Delta}$ acting on this model class.
The following lemma is an obvious consequence of their results.

\begin{lemma}\label{lem2.1}
Let $M=\mathfrak{R}^n \sharp\, {\R}^m$ with $2<n \leq m$ and  $\Delta$ be
 the Laplace-Belltrami operator acting on $M$. Then the  kernel $p_t(x,y)$ of the semigroup  $e^{-t\Delta}$  satisfies the following on-diagonal estimate
$$\sup_{x\in M} p_t(x,x) \le C\Big( t^{-n/2}+t^{-m/2}\Big).$$
\end{lemma}
\begin{proof}
The proof is a straightforward consequence of \cite[Corollary 4.16 and Section 4.5]{GS}. In fact the estimate in 
Lemma~\ref{lem2.1} are essentially more elementary than the quoted results and can be verified directly.  
\end{proof}

 Lemma~\ref{lem2.1} implies the following result. 

\begin{lemma}\label{le2.2}
 Let $M=\mathfrak{R}^n \sharp\, \R^m$ with $2<n \le m$ and  $\Delta$ be
   the Laplace-Belltrami operator acting on $M$.
Then for any $\sigma>m/4$ and $\kappa=(m-n)/4$ estimate \eqref{a1} holds for $\Delta$, that is 
there exists constant $C$  
such that
\begin{eqnarray*}
\|V^{1/2}_t(I+t^2\Delta)^{-\sigma}\|_{2 \to \infty} \le C (1+t^2)^{(m-n)/4}
\end{eqnarray*}
for all $t>0$.
\end{lemma}

\begin{proof} For every  $\sigma>0$
\begin{eqnarray}\label{ee}
{(I+t^2\Delta)^{-2\sigma}}
= \frac{1}{\Gamma(2\sigma)}\int_0^\infty e^{-s} \, s^{2\sigma-1}\exp(-st^2\Delta)\,ds.
\end{eqnarray}
Hence  if $\sigma>m/4,$ then by Lemma~\ref{lem2.1}
\begin{eqnarray*}
|K_{(I+t^2\Delta)^{-2\sigma}}(x,x)| &\le& C
\int_0^\infty e^{-s} \, s^{2\sigma-1}\left((st^2)^{-n/2}+(st^2)^{-m/2}\right) ds\\
&\le& C\big(t^{-n}+t^{-m}\big)
\end{eqnarray*}
for all $x\in M$ and $t>0$.

Recall that $V(x,t)\leq Ct^{m}$ for all $x\in M$ and $t>0$ so by duality and the above inequality
\begin{eqnarray*}
\|V^{1/2}_t(I+t^2\Delta)^{-\sigma}\|^2_{2 \to \infty}
&=&\sup_{x\in M} V(x,t) \int_M |K_{(I+t^2\Delta)^{-\sigma}}(x,y)|^2d\mu(y)\\
&\leq&C\sup_{x\in M} V(x,t)  |K_{(I+t^2\Delta)^{-2\sigma}}(y,y)|\\
&\leq&C (1+t^2)^{(m-n)/2}.
\end{eqnarray*}
This proves Lemma~\ref{le2.2}.
\end{proof}

Theorems~\ref{thm1.1}, \ref{th4.2} and Lemma~\ref{le2.2} yield the following theorem.

\begin{theorem}
Let $M=\mathfrak{R}^n \sharp\, {\R}^m$ with $2<n  \le  m$.
Suppose that $\Delta$ is  the  Laplace-Belltrami operator acting on $M$.
Then estimate \eqref{e3.55} holds
for any exponent $\sigma>m/4$ and $\kappa=(m-n)/4$ and   {\rm Spec}$(\Delta_p)$ - the spectrum of the operator $\Delta_p$ is independent of $p$
for all $1 \le p < \infty$. 

Moreover, if $F$ is a bounded Borel function such that $\supp F\subset [-1, 1] $
and $F\in  H^s(\mathbb R)$ for some $s>m/2 +(m-n)/2+1$,
then the operator $F(\Delta)$ is bounded on $L^p(M)$ and
\eqref{e3.5555} holds with $\kappa=(m-n)/4$ that is 
 \begin{eqnarray*}
\|F(t\Delta)\|_{p\to p} \leq C(1+t)^{(m-n)/4}\|F\|_{H^{s}}
\end{eqnarray*}
for all $t>0$.
\end{theorem}

\begin{remark}\label{rem5}
{\rm Using the estimates from \cite{GS} one can show that $\|\exp(-t\Delta)\|_{p\to p} \le C$ uniformly in $t$ and~$p$. 
Therefore we conjecture that one can strengthen all the above results concerning the manifolds with ends by taking $\kappa=0$
instead of $(m-n)/4$.   }
\end{remark}

\medskip

\subsection{Semigroups without uniform $L^p$ bounds - Schr\"odinger operators with strongly subcritical potentials }\label{sub5.3}
In this subsection, we consider Schr\"odinger operators $L=-\Delta+V$ on $\mathbb{R}^n$, $n\geq 3$ and 
we assume that $L\geq 0$. 
Let $V=V_+-V_-$ be  the decomposition of $V$ into its positive and negative parts.  We say 
$L$ (or $V$) is  strongly subcritical, if  there exists small enough 
$\varepsilon>0$ such that
$$
L-\varepsilon V_-\geq 0.
$$

Following Davies and Simon~\cite{DS}, we say that  $L$ has a  resonance $\eta$  if there exists a non-zero function $\eta$ such that $L\eta=0$ and then we say that it is slowly varying with 
 index $\alpha$ for some  $0<\alpha<(n-2)/2$ if 
there exists  a constant $C>0$ such that
$$
\frac{\eta(x)}{\eta(y)}\leq C(1+|x-y|)^\alpha
$$
for all $x,y\in \R^n$.

 The following example of Schr\"odinger operators with a slowly varying resonance comes from  Murata~\cite{M1,M2}.
 Let $n\geq 3$ and put $L=-\Delta+V$ where $V(x)=0$ if $|x|\leq 1$ and $V(x)=-c/|x|^2$ if $|x|>1$ with constant $0<c\leq ((n-2)/2)^2$.
Then $L\geq 0$ and $V$ is strongly subcritical. Also $L$ has one positive radial resonance $\eta$ satisfying $\eta(x)\sim |x|^{-\alpha}$
as $|x|\to \infty$. Here
$$
0<\alpha=\frac{n-2}{2}-\sqrt{\left(\frac{n-2}{2}\right)^2-c}<\frac{n-2}{2}.
$$

The following theorem is a consequence of Theorems~\ref{thm1.1} and  \ref{th4.2} applied to the setting considered 
in~\cite{DS}.
\begin{theorem}\label{coro5.10}
Let $L=-\Delta+V$ be Schr\"odinger operator on $\mathbb{R}^n$, $n\geq 3$.
Assume that $L\geq 0$, that $V$ is strongly subcritical and that $L$ has a resonance $\eta\geq 0$
in $L_w^{n/\alpha}$ which is slowly varying with index $\alpha$ where $0<\alpha<(n-2)/2$.
Then there exists a constant $C=C(\sigma, \varepsilon)>0$ such that estimate \eqref{e3.55} holds
for any $\sigma>n/4$ and $\kappa=\alpha/2+\varepsilon$ and {\rm Spec}$(L_p)$ the spectrum of the operator $L_p$ is independent of the space $L^p(\R^n)$
for all $1 \le p < \infty$. 

Moreover, if $F$ is a bounded Borel function such that $\supp F\subset [-1, 1] $
and $F\in  H^s(\mathbb R)$ for some $s>n/2+1$,
then the operator $F(L)$ is bounded on $L^p(\R^n)$ and
\eqref{e3.5555} holds with $\kappa=\alpha/2+\varepsilon$ that is 
\begin{eqnarray*}
\|F(tL)\|_{p\to p} \leq C(1+t)^{\alpha/2+\varepsilon}\|F\|_{H^{s}}
\end{eqnarray*}
for all $t>0$.
\end{theorem}
\begin{proof}
By \cite[Theorem 3.3]{CouS} the semigroup generated by operator $L$
satisfies Davies-Gaffney estimates \eqref{dg} and so $L$ satisfies the finite speed propagation property. 
Next  by the result obtained by Davies and Simon, see \cite[Theorem 14]{DS}, for any $\varepsilon>0$
\begin{eqnarray}\label{e5.1}
c_{1,\varepsilon}t^{-n/4}(1+t)^{\alpha/2-\varepsilon}
\leq \|e^{-tL}\|_{2\to \infty}\leq c_{2,\varepsilon} t^{-n/4}(1+t)^{\alpha/2+\varepsilon}.
\end{eqnarray}
Now it follows from  the above estimate and Lemma~\ref{prop2} that for any $\sigma>n/4,\varepsilon>0$ 
and $\kappa=\alpha/2+\varepsilon$ resolvent estimate \eqref{a1} holds. That is 
there exists a constant $C=C(\sigma,\varepsilon)>0$ such that
$$
\|V_t^{1/2}(I+t^2L)^{-\sigma}\|_{2\to \infty}\leq C (1+t^2)^{\alpha/2+\varepsilon}.
$$
Thus the theorem follows from Theorems~\ref{thm1.1} and \ref{th4.2}.
\end{proof}

\begin{remark}\label{rm5.19}{\rm 
(A). In contrast to Remark \ref{rem5} we expect that one cannot remove the additional term $(1+t)^{\alpha/2+\varepsilon}$ especially for $p$ close to $1$ or $\infty$. Otherwise the estimate from Theorem \ref{coro5.10} with $\kappa=0$ would imply the uniform bounds for the semigroup which contradicts lower estimates for the semigroup in  \cite[Theorem 15]{DS}.
However, for some range of $p$ close to $2$, the uniform version with $\kappa=0$ holds, see \cite[Example 4.17]{CouS}. We do not discuss it here.

(B). The same approach can be used to study generalised Schr\"{o}dinger operators
$$\mathcal{L}=\nabla^*\nabla + \mathcal{R},$$
acting on a finite-dimensional Riemannian  bundle $E\rightarrow M$. Here $\nabla$ is a connection on $E\rightarrow M$ which is  compatible with the metric, and $\nabla^*\nabla$ is the so-called ``rough Laplacian''. Such operators 
include  Hodge Laplacians $\vec{\Delta}=dd^*+d^*d$ in the case where $ \mathcal{R}$ is Ricci curvature,  see e.g. \cite{CouZ, CDS}.  One can expect that in the subcritial case it leads to a similar discussion as described in this section. We do not study these operators here.}
\end{remark}

\medskip

\subsection{Non-uniform Gaussian upper bounds - Davies' example} \label{sub5.2}
The following example comes from~\cite{Da}. 
In this section we assume that
\begin{eqnarray}\label{e5.21}
V_r(x)\leq \left\{
\begin{array}{ll}
C r^{n_1} & {\rm if}\ \  0<r\leq 1\\[4pt]
C r^{n_2} & {\rm if}\ \  1\leq r<\infty,
\end{array}
\right.
\end{eqnarray}
where  $0<C<\infty$, $0<n_1\leq n_2<\infty$ and $V_r(x)=\mu(B(x,r))$. We want to stress here that we do not
assume that the doubling condition \eqref{dc} holds . 
Let $L$ be a non-negative self-adjoint operator acting on $L^2(X)$ such that an integral kernel $p_t(x,y)$
of the semigroup $e^{-tL}$  satisfies
\begin{eqnarray}\label{e5.22}
|p_t(x,y)|\leq C t^{-n_1/2}\exp\left\{-c{d(x,y)^2\over t}\right\}.
\end{eqnarray}

Considering a new metric $d'=2\sqrt{c}d$ we can assume that $c=1/4$ in the above estimate.
Note that \eqref{e5.21} still holds possibly with a new constant $C$. In what follows we always assume 
that $c=1/4$ in~\eqref{e5.22}.

First we claim that $L$ satisfies the finite propagation speed property for 
the corresponding wave equation  or equivalently the semigroup
$e^{-tL}$ satisfies Davies-Gaffney estimates. Recall that the semigroup satisfies Davies-Gaffney 
estimates if 
for all $f\in L^2(B_1,\mu)$ and
$g\in L^2(B_2,\mu)$
\begin{equation}\label{dg}\tag{DG}
\left|\langle  e^{-tL}f,g \rangle \right| \leq e^{-\frac{r^2}{4t}}\|f\|_2\|g\|_2,
\end{equation}
where $B_1$ and $B_2$ are open balls in $X$  and $d(B_1,B_2)>r>0$.

Now we are able to state the following lemma.

\begin{lemma}
Suppose that the space $(X, d, \mu)$ and the operator $L$ satisfies conditions \eqref{e5.21} and \eqref{e5.22}.
Then $L$ satisfies Davies-Gaffney estimates \eqref{dg} and as a consequence the finite propagation speed property
\eqref{fs} holds as well. 
\end{lemma}
\begin{proof}
Let  $f\in L^2(B_1,\mu)$ and
$g\in L^2(B_2,\mu)$, where $B_1=B(x_1,r_1)$ and $B_2=B(x_2,r_2)$ are open balls in $X$  and $d(B_1,B_2)>r>0$. Set  $U_k(x_1,r_1)=B(x_1,2^{k+1}(r_1+r))\backslash B(x_1,2^{k}(r_1+r))$.
Decompose 
$$X\setminus B(x_1,r_1+r)= \bigcup_{k=0}^\infty U_k(x_1,r_1) \quad \mbox{and} \quad X\setminus B(x_2,r_1+r)=\bigcup_{k=0}^\infty U_k(x_2,r_2).$$
Then by interpolation
\begin{eqnarray}\label{e55.44}
\|e^{-tL}\|_{L^2(B_1)\to L^2(B_2) }&\leq&
\|e^{-tL}\|_{L^1(B_1)\to L^1(B_2) }+\|e^{-tL}\|_{L^\infty(B_1)\to L^\infty(B_2) }\nonumber\\ 
&\leq&\sup_{x\in B_1}\int_{B_2} |p_t(x,y)|d\mu(x)+
\sup_{x\in B_2}\int_{B_1} |p_t(x,y)|d\mu(y)\nonumber\\
&\leq&\sup_{x\in B_1}\sum_{k=0}^\infty \int_{U_k(x_1,r_1)} |p_t(x,y)|d\mu(x)+
\sup_{x\in B_2}\sum_{k=0}^\infty \int_{U_k(x_2,r_2)} |p_t(x,y)|d\mu(y)\\
&\leq& C(1+t)^{(n_2-n_1)/2}e^{-\frac{r^2}{4t}}\nonumber,
\end{eqnarray}
which implies
$$
\left|\langle  e^{-tL}f,g \rangle \right| \leq C(1+t)^{(n_2-n_1)/2}e^{-\frac{r^2}{4t}}\|f\|_2\|g\|_2.
$$
Then by \cite[Lemma 3.2]{CouS}, the self-improving property for Davies-Gaffney estimates, the above estimate implies the Davies-Gaffney estimates \eqref{dg}.
\end{proof}

Next we would like to make the following observation. 
\begin{lemma}\label{le2.222}
 Let $X$ and $L$ satisfy conditions~\eqref{e5.21} and \eqref{e5.22}. Then for any $\sigma>n_1/4$ and $\kappa=(n_2-n_1)/4$ resolvent estimate \eqref{a1} holds for $L$. 
\end{lemma}
\begin{proof}
The proof is similar to that of Lemma~\ref{le2.2}. We omit the details.
\end{proof}

The following theorem follows from  Theorem~\ref{thm1.1}, \ref{th4.2} and Lemma~\ref{le2.222}.

\begin{theorem}\label{coro5.5}
 Let $X$ and $L$ satisfy conditions~\eqref{e5.21} and \eqref{e5.22}.
Then there exists a constant $C=C(\sigma)>0$ such that estimate \eqref{e3.55} holds
for any $\sigma>n_1/4$ and $\kappa=(n_2-n_1)/4$ and {\rm Spec}$(L_p)$ the spectrum of the operator $L_p$ is independent of the space $L^p(X)$
for all $1 \le p < \infty$.

Moreover, if $F$ is a bounded Borel function such that $\supp F\subset [-1, 1] $
and $F\in  H^s(\mathbb R)$ for some $s>n_2/2+1$,
then the operator $F(L)$ is bounded on $L^p(X)$ and
\eqref{e3.5555} holds with $\kappa=(n_2-n_1)/4$, that is 
 \begin{eqnarray*}
\|F(t\Delta)\|_{p\to p} \leq C(1+t)^{(n_2-n_1)/4}\|F\|_{H^{s}}
\end{eqnarray*}
for all $t>0$.

\end{theorem}

\begin{remark}
{\rm In \cite[Lemma 2]{Da} Davies proved that under assumptions \eqref{e5.21} and \eqref{e5.22} the following 
estimate holds 
 \begin{eqnarray}\label{e3.6}
\|e^{i\xi tL}e^{-t L}\|_{1\to 1}\leq C(1+\xi^2)^{n_2\over 2 } e^{\epsilon t}
\end{eqnarray}
for all $t>0$ and $\xi\in \mathbb{R}$.
Theorem~\ref{coro5.5} improves the above estimates  and shows that 
 \begin{eqnarray}
\|e^{i\xi tL}e^{-t L}\|_{1\to 1}\leq C(1+\xi^2)^{s } (1+t)^{(n_2-n_1)/4}
\end{eqnarray}
for any $s>n_2/4+1/4$. As a consequence 
the estimates from Lemma~\ref{le4.1} improve the complex time resolvent estimates described in 
\cite[Lemma 3]{Da}.
}
\end{remark}

\medskip

\subsection{Ambient spaces satisfying the  doubling condition }\label{sub5.4}
Let $(X,d,\mu)$ be a metric measure space satisfying doubling condition with homogenous dimension $n$.


%

The following statement is a direct consequence of Corollary 4.16 of \cite{CouS}.

\begin{proposition}\label{coro5.12}
Let $X$ satisfy the doubling condition \eqref{vd} with exponent $n$.
Suppose next the wave propagator corresponding to $L$ satisfies finite speed property \eqref{fs} and 
that condition \eqref{a1} holds for some $\sigma >0$ and $\kappa=0$ that is 
\begin{eqnarray*}
\|V^{1/2}_t(I+t^2L)^{-\sigma}\|_{2 \to \infty} \le C.
\end{eqnarray*}
  
Then   there exists a constant $C>0$ such that
 \begin{eqnarray}\label{e3.555}
\|e^{i\xi tL}e^{-t L}\|_{1\to 1}\leq C(1+\xi^2)^{n/4} 
\end{eqnarray}
for all $t>0$ and $\xi\in{\R}$.
\end{proposition}

\begin{proof}
Estimates \eqref{e3.555} are proved in \cite[Corollary 4.16]{CouS} under condition 
$\|V^{1/2}_t e^{-t^2L}\|_{2 \to \infty}$.
However 
$$
\|V^{1/2}_t e^{-t^2L}\|_{2 \to \infty} \le \|V^{1/2}_t(I+t^2L)^{-\sigma}\|_{2 \to \infty}  
\|e^{-t^2L}(I+t^2L)^{\sigma} \|_{2 \to 2}  \le C \|V^{1/2}_t(I+t^2L)^{-\sigma}\|_{2 \to \infty}
$$
compare  Lemma~\ref{prop2}. Hence Proposition \ref{coro5.12} follows directly from \cite[Corollary 4.16]{CouS}.
\end{proof}

\begin{remark}\label{rm5.112}

{\rm (A). The direct application of Theorem \ref{thm1.1} yields the estimate \eqref{e3.555} with exponent 
$\sigma +1/4$ instead of $n/4$. Note that one usually expect that $\sigma > n/4$, see Lemma~\ref{prop2}.
However the proof of Theorem \ref{thm1.1} is still valid if doubling condition fails and is essentially less complex than the proof of  \cite[Corollary 4.16]{CouS}. See also \cite[Theorem 7.3, Proposition 7.9]{O}.

(B). In the doubling setting the estimates $\|V^{1/2}_t e^{-t^2L}\|_{2 \to \infty} \le C $ and the finite 
speed propagation property \eqref{fs} are equivalent to the standard Gaussian upper bound for the heat kernel  corresponding  to the operator $L$
\begin{eqnarray}\label{e5.41}
|p_t(x,y)|\leq C \frac{1}{V(y,\sqrt t)}\exp\left\{-c{d(x,y)^2\over t}\right\}
\end{eqnarray}
see  e.g. \cite{Si}. Therefore  Proposition~\ref{coro5.12} can be reformulated in term of Gaussian upper bound
assumption. 

(C). Again  in the doubling setting, the similar results can be obtained using the techniques from \cite{CowS, DOS, COSY} but the proof is also more complex than our approach. Vice versa, Theorem~\ref{thm1.1} yields alternative proof 
of compactly supported version of the results described in \cite{CowS, DOS, COSY}. However the number of 
required derivatives will be slightly bigger in this approach. 
}
\end{remark}

\medskip

 \noindent
{\bf Acknowledgments.}
 The authors would like to thank X.T. Duong and J. Li for helpful discussions.
P. Chen was partially  supported by NNSF of China 11501583, Guangdong Natural Science Foundation 2016A030313351 and the Fundamental Research Funds for the Central Universities 161gpy45.  A.~Sikora is supported by
Australian Research Council  Discovery Grant DP 130101302.
 L.X.~Yan is supported by the NNSF
of China, Grant Nos.~11371378 and  ~11521101).

\end{document}